\RequirePackage{fix-cm}
\documentclass[smallextended]{svjour3}       
\smartqed  
\usepackage{mathptmx}      
\usepackage{eucal}

\usepackage{amssymb}
\usepackage{amsmath}
\usepackage[colorlinks,citecolor=blue,urlcolor=blue]{hyperref}

\usepackage{tikz}
\usetikzlibrary{decorations.markings}
\usepackage{pgfplots}

\newcommand{\Real}{\mathbb R}
\newcommand{\eps}{\varepsilon}

\newcommand{\one}[1]{  {1}_{\{#1\}}}

\renewcommand{\P}{\mathbb{P}}
\newcommand{\E}{\mathbb{E}}
\newcommand{\var}{\mathrm{Var}}
\newcommand{\Z}{\mathbf{Z}}
\newcommand{\X}{\mathbf{X}}
\newcommand{\Y}{\mathbf{Y}}
\newcommand{\U}{\mathbf{U}}
\renewcommand{\L}{\mathbf{L}}
\newcommand{\x}{\mathbf{x}}
\newcommand{\y}{\mathbf{y}}

\newcommand{\f}{\mathbf{f}}
\newcommand{\g}{\mathbf{g}}
\newcommand{\h}{\mathbf{H}}
\newcommand{\bdelta}{\boldsymbol{\delta}}
\newcommand{\overbar}[1]{\mkern 1.5mu\overline{\mkern-1.5mu#1\mkern-1.5mu}\mkern 1.5mu}

\journalname{Journal of Mathematical Biology}
\begin{document}

\title{On the establishment of a mutant
\thanks{This work has been supported by the Australian Research Council Grant DP150103588.}
}
%



\author{Jeremy Baker  \and
        Pavel Chigansky \and
        Peter Jagers \and
        Fima Klebaner               
}

\authorrunning{J. Baker et al} 

\institute{J. Baker \at
              School of Mathematics,
              Monash University, Monash, VIC 3800, Australia \\
              \email{jeremy.baker@monash.edu}           
           \and
           P. Chigansky \at
              Department of Statistics, 
              The Hebrew University, Mount Scopus, 
              Jerusalem 91905, Israel \\
              \email{Pavel.Chigansky@mail.huji.ac.il}
           \and
              P. Jagers \at
              Mathematical Sciences, 
              Chalmers University of Technology and \\
              University of Gothenburg, 
              SE-412 96 Gothenburg, Sweden \\
              \email{jagers@chalmers.se}
           \and 
              F. Klebaner \at
              School of Mathematics,
              Monash University, 
              Monash, VIC 3800, 
              Australia \\
              \email{fima.klebaner@monash.edu}       
}

\date{Received: date / Accepted: date}

\maketitle

\begin{abstract}
How long does it take for an initially advantageous mutant to 
establish itself in a resident population,  and what does the
population composition look like then? We approach these questions in
the framework of the so called Bare Bones evolution model \cite{KleVa}
that provides a simplified approach to the adaptive population
dynamics of binary splitting cells. As the mutant 
population grows, cell division becomes less probable,
and it  may in fact turn less likely than that of residents. 

Our analysis rests on the assumption of the process starting from
resident population, with sizes proportional to a large carrying
capacity $K$. Actually, we assume carrying capacities to be
$a_1K$ and $a_2K$ for the resident and the mutant
populations, respectively, and study the dynamics for $K\to\infty$. We
find conditions for the
mutant to be successful in establishing itself  alongside the
resident. The time it takes turns out to be  proportional to $\log
K$. We introduce the time of establishment through the  asymptotic behaviour 
of the stochastic nonlinear dynamics describing the evolution, and
show that it is indeed $\log K/\log \rho$, where $\rho>1$ is twice
the probability of successful division of the mutant at its
appearance. Looking at the composition of the population, 
at times $\log K/\log \rho +n, n \in \mathbb{Z}_+$, we find that
the densities (i.e. sizes relative to carrying capacities) of  
both populations follow closely the corresponding two dimensional
nonlinear deterministic dynamics  that starts at  
{\it a random point}. We characterise this random initial condition
in terms of the scaling limit of the corresponding 
dynamics, and the limit of the properly scaled initial binary splitting process of the mutant. 
The deterministic approximation with random initial condition is in fact valid asymptotically
at all times $\log K/\log \rho +n$ with $n\in \mathbb{Z}$.

\end{abstract}

\section{Introduction}
There has been much work in stochastic adaptive dynamics and evolutionary branching, see 
\cite{DL96}, \cite{Metz96}, \cite{CFA}, \cite{CM}, \cite{Serik}, to
mention just a few. Here we confine ourselves to a simple mathematical
model for evolution, where an established resident population is invaded by a mutant. 
From that moment on, the two populations compete for resources.  At
the moment of invasion the resident, wild-type, population is assumed
to have the size near its carrying capacity $a_1 K$. Here $K$ should be
thought of as large, and $a_1>0$ is fixed. The size of the mutant
population is initially negligible as compared to $K$, since it starts
from one individual. It has a reproductive advantage over the
resident, but as its progeny grows this advantage diminishes.
 
We want to answer the question of how long it takes for a mutant to
become established, i.e. to grow to a size comparable to the host
population. And what is the population composition then? 
Already the simplified model of two competing populations we consider,
will require new mathematical techniques and lead to insightful results.
We show that the deterministic approximation with a random initial condition 
is valid for times $[\log K/\log \rho] +n$ with any fixed $n\in \mathbb{Z}$ and a large $K$.
However, unlike in the classical case on deterministic approximation, \cite{Ku70} and \cite{B79}, 
some stochasticity remains and enters as a random initial condition. 

\subsection{The Bare Bones evolutionary model.} 
This simple but basic model of species reproducing under interaction
with their environment was  introduced in \cite{KleVa}. It builds upon
asexual binary splitting and evolves in discrete time. Thus, each 
individual either gets two children in the next generation or none.
However, interaction with environment and population size is allowed - in
contrast to classical stochastic approaches - but drastically condensed.
Following the idea of Malthus, populations reach sizes proportional
to available resources, and we assume that the 
the habitat is characterised by a {\em carrying capacity}, $K>0$,
thought of as large. Given the population size, individuals reproduce independently. 
 Initially only the resident, wild-type,  population is present and,
 at population size $z$, the individual probability of successful splitting  
is taken to be $a_1 K/(a_1 K + z)$. Here $a_1$ is a constant, which
determines the population size at its
macroscopic (quasi-)equilibrium: when $z=a_1 K$, the probability of
splitting is $1/2$. On the average, thus, a population of this size
produces one child per individual. As a result, the population size
fluctuates around this (quasi-)steady state for what is presumably a
very long time, cf. \cite{JKling}.

In that stage, the population will experience its first mutation giving rise to
a new population. The new, mutant population starts from a single individual,
its ancestor. The basis of adaptive dynamics can then be said to be
furnished by the branching mechanism, which forces the new population
to either die out or else embark on exponential growth, in
which case the old resident dies out, or the two populations will
coexist for a time span that turns out to be exponential in the carrying capacity.

Mathematically, this dynamics can be described as follows. 
The branching process starts from a pair of positive integers
$\Z(0)=\big(Z_1(0),Z_2(0)\big)$, the first component denoting
the size of the resident and the second that of the mutant population,
at time 0, when the mutation appears. We assume that the established
original population is at equilibrium at the moment of invasion $n=0$,
$Z_1(0)=[a_1K]$, and  $Z_2(0)=1$. Each population
develops by binary splitting with probabilities dependent on the numbers of cells, with
transitions from generation $n$ to $n+1$ described by the recursion 
\begin{equation}\label{Zprocess}
{\Z}(n+1)=\Big( Z_1(n+1),Z_2(n+1)\Big) =\bigg(
\sum_{k=1}^{Z_{1}(n)}\xi_1 (n,k),
\sum_{k=1}^{Z_{2}(n)}\xi_2 (n,k)\bigg).
\end{equation}%
The random variables $\xi_i(n,k)\in \{0,2\}$ are 
independent, given the preceding, and only depend upon the last
generation $\Z(n)$, with probabilities 
\begin{equation}\label{distr}
\begin{aligned}
\P\Big( \xi_1(n,k)=2|  \Z(n)\Big) &=&%
\frac{a_{1}K}{a_{1}K+Z_1(n)+\gamma Z_2(n)},\\
\P\Big( \xi_2(n,k)=2|  \Z(n)\Big)  &=&
\frac{a_{2}K}{a_{2}K+\gamma Z_1(n)+Z_2(n)},
\end{aligned}
\end{equation}
where $a_2>0$ is the parameter, which controls the mutant equilibrium population size, and $\gamma$ is 
the interaction coefficient, assumed to satisfy
$0<\gamma<1$. The biological meaning of $\gamma$ is that cells of one
type encroach less upon the reproduction of the other cell type than do cells of the same
type. That $\gamma$ is the same  in both probabilities means that influence is symmetric between the cell types. 

In the absence of mutants, the established population thus has a
critical reproduction, whereas the mutant population starts
supercritically, provided $a_2>\gamma a_1$, as is assumed throughout the paper, see \eqref{coex} below.


\subsection{Stochastic nonlinear dynamics for the evolution of the density.}
Important insights  into the  behaviour of populations with state
dependent reproduction and large carrying capacity is provided by
their {\em density process}, \cite{Kleb84}, \cite{Kleb93}. It allows
representation of the process as having stochastic nonlinear dynamics,
which can be separated 
into a deterministic part and a random  perturbation. 
This is useful not only for the mathematical analysis but also for
the biological interpretation.

The density process is the population sizes relative to $K$
$$
\X(n)=\big(X_1(n),X_2(n)\big):=\big(Z_1(n)/K,Z_2(n)/K\big).
$$
Note that the splitting probabilities (and hence the
offspring distributions) in \eqref{distr} are in fact
functions of the density; denoting the  density state by $\x=(x_1,x_2)$ we see that 
\begin{align*}
\P\left(\xi_1(n,k)=2|  \X(n) =  \x\right) &= \frac{a_{1}}{a_{1}+x_{1}+\gamma x_{2}}, \\
\P\left(\xi_2(n,k)=2|  \X(n) =  \x\right)& =\frac{a_{2}}{a_{2}+\gamma x_{1}+x_{2}}.
\end{align*}
Accordingly,  the offspring  mean $\mathbf{m}(\x)= \big(m_1(  {\x}),
m_2(  {\x})\big)$  at $\x$  is also a function of the density
\begin{align*}
m_{1}(\x) &=\E\big(\xi_1(n,k)|  \X(n)=  {\x}\big)=\frac{2a_{1}}{a_{1}+x_{1}+\gamma x_{2}}, \\
m_{2}(\x) &=\E\big(\xi_2(n,k)|  \X(n)=  {\x}\big)=\frac{2a_{2}}{a_{2}+\gamma x_{1}+x_{2}}.
\end{align*}

The underlying deterministic dynamics 
\begin{equation}  \label{map2d}
\x(n+1)= \f \big( \x(n)\big),
\end{equation}
is determined by the function 
$\f(\x)=\big(f_1(\x), f_2(\x)\big)$,
\begin{equation}\label{dynf}
\begin{aligned}
f_1( \x ) &= x_1 m_1(\x)= \frac{2x_1a_1}{a_1+x_1+\gamma x_2}, \\
f_2(\x) &= x_2 m_2(\x) = \frac{2x_2a_2}{a_2+\gamma x_1+ x_2}.
\end{aligned}
\end{equation}
This can be easily seen from \eqref{Zprocess} by writing the density process as
\begin{equation}\label{density2d}
\begin{aligned}
X_1(n+1) &=  X_1(n)m_{1}\big( \X(n)\big)   + \frac{1}{K}\sum_{j=1}^{KX_1(n)}\Big(\xi_1(n,j)-m_{1}\big(   \X (n)\big)\Big)  
 \\
X_2(n+1) &=   X_2(n) m_2\big(\X(n)\big) + \frac{1}{K}\sum_{j=1}^{KX_2(n)}\Big(\xi_2(n,j) -m_{2}\big(  \X(n) \big)\Big).
\end{aligned}
\end{equation}
The first term on the r.h.s. of \eqref{density2d} gives the deterministic dynamics \eqref{map2d}, 
and the second term acts as the random perturbation, 
\begin{equation}\label{density2d1}
\X(n+1) =\f\big(\X(n)\big)+\frac{1}{\sqrt{K}} \boldsymbol{\eta} (n+1,\X_n),
\end{equation}
with 
$$
\eta_i(n+1, \x)=\frac{1}{\sqrt K}\sum_{j=1}^{K x_i}\big(\xi_i(n,j)-m_i(\x)\big),\quad i=1,2.
$$
These random variables have zero mean and variance $4x_ip_i(\x)\big(1-p_i(\x)\big)$, where $p_i(\x)$ are the splitting probabilities. 
Therefore the random noise term  in \eqref{density2d1} is of order
$1/\sqrt K$ and the density process can indeed be viewed as generated
by a nonlinear dynamical system, perturbed by a small random disturbance.

Note that in the view of the above discussion, the trajectory of the deterministic system \eqref{map2d} 
depends on $K$ through the initial condition  $\x^K(0)=\big(\frac{[a_1K]}{K}, \frac 1 K\big)$. Similarly, the process generated by the stochastic 
dynamics \eqref{density2d1}, depends on $K$ through  $\X^K(0)=\big(\frac{[a_1K]}{K}, \frac 1 K\big)$ and the noise term. Whenever appropriate, we will leave 
this dependence implicit, omitting it from the notation.

\subsection{Deterministic dynamics }
If we neglect the small random noise in \eqref{density2d1}, we obtain the deterministic dynamics \eqref{map2d}. 
Fixed points (solutions to $\f(\x)= \x$) play an important role in the behaviour of such systems. The trajectories are repelled from the unstable fixed points and attracted to the stable ones.   Our system, generated by the function $\f(\cdot)$ in \eqref{dynf}, has four fixed points,
\begin{equation}\label{stableFP}
\begin{aligned}
&\x^{\mathrm{ex}}=(0,0) & &\text{(total extinction equilibrium)}\\
&\x^{\mathrm{re}}=(a_1,0) & & \text{(resident equilibrium in absence of mutant)} \\
&\x^{\mathrm{mu}} = (0,a_2) & & \text{(mutant equilibrium in absence of resident)} \\
&\x^{\mathrm{co}} =\Big(\frac{a_1-\gamma a_2}{1-\gamma^2}, \frac{a_2-\gamma a_1}{1-\gamma^2}\Big) & & \text{(coexistence equilibrium)}
\end{aligned}
\end{equation}
Since we are concerned with both populations, the relevant case is
when both coordinates of $\x^{\mathrm{co}}$ are nonnegative. This is true if the
following {\em co-existence} condition holds
\begin{equation}\tag{C}\label{coex}
a_1-\gamma a_2>0,\;\;a_2-\gamma a_1>0.
\end{equation}

It is easy to see by examining  the Jacobian matrix $\nabla \f(\x)$, see \eqref{Adef} below, that the  point 
$\x^{\mathrm{co}}$ is stable, and $\x^{\mathrm{ex}}$  unstable. The points $\x^{\mathrm{re}}$ and $\x^{\mathrm{mu}}$ are saddle points, 
that is, stable in one direction and unstable in another. In our theory the point 
$\x^{\mathrm{re}}=(a_1,0)$ plays a special role due to proximity of
the initial condition  $\big(\frac {[a_1 K]} K,\frac 1 K\big)$. In the absence of a mutant, $a_1$
is the stable equilibrium for the resident population, and $0$ is unstable for the mutant population. 

\subsection{The large capacity limit of the stochastic dynamics.}
A rigorous treatment for neglecting small noise is given by 
the classical results in perturbation theory of dynamical systems, see
e.g.  \cite{Ku70}, \cite{B79}, \cite{FW84}, \cite{Kif88}.   
They assert  that as the noise converges to zero, that is, when  $K\to\infty$,
the trajectory $\X^K(n)$ of the stochastic system \eqref{density2d1} converges on any {\em bounded} time interval to that of the
deterministic dynamics  \eqref{map2d}, started from the initial condition $\x(0) = \lim_{K\to \infty}\X^K(0)$.
Namely, for an arbitrary but fixed integer $N$,
\begin{equation}\label{clim}
\max_{n \le N} \big|\X^K(n)-\x(n)\big|\xrightarrow[K\to\infty]{\P}0.
\end{equation}

In our setup,    the initial condition turns out to be the fixed point $\x^{\mathrm{re}}$,
$$
\x(0) =\lim_{K\to\infty} \X^K(0)=\lim_{K\to\infty} \big(\tfrac{[a_1 K]} K,\tfrac 1 K\big)= (a_1,0)=\x^{\mathrm{re}}.
$$  
Therefore the corresponding limit trajectory is constant,  $\x(n)=\x^{\mathrm{re}}$ for all $n=1,2,...$
Consequently, the limit \eqref{clim} fails to provide any information on the transition 
to a new coexistence equilibrium. We shall see that if such a transition occurs, it
becomes visible much later, at a  time increasing with $K$, in fact, of order $\log K$.

Recently, limit theorems, capable of capturing this transition, were obtained in
\cite{BHKK15}, \cite{BCK16}, \cite{CJK18}, \cite{BCHK}. They involve 
a time shift which grows logarithmically in $K$. 
In \cite{BHKK15} this shift is {\em random} and the process $\X^K(n)$ is
approximated by the trajectory of the deterministic system \eqref{map2d} with a random shift. We have learnt from a referee that
a precursor to random shift theory in \cite{BHKK15} in the context of epidemic models can be found in \cite{Metz}, where    precise conjectures  were stated and later proved in an unpublished manuscript  for the simple SIR epidemic model, \cite{Altman}, \cite{Mollison}.  

In \cite{BCK16}, \cite{CJK18}, \cite{BCHK}, the shift is {\em deterministic}, and  $\X^K(n)$ converges to a trajectory of 
\eqref{map2d},  started from a {\em random} initial condition.  
While the two approaches, the random shift and the random initial condition, are related, they  are not equivalent.  
The main building block in the random initial condition theory  is a
certain scaling limit of the deterministic flow, which
does not appear in the random shift theory. Existence of this limit
was so far established only in the one dimensional case. 

This work is
the first such result in two dimensions. Having established it, we
can complement the ``random shift'' picture in \cite{BHKK15} with that of ``random initial condition'' for the Bare Bones model.  Recently heuristics for similar random initial conditions for selective sweeps in large populations in one dimension were given in \cite{ML}. Other stochastic approaches involving 
carrying capacity can be found in \cite{L05}, \cite{L06}.

\section{Main results}
In what follows we consider the stochastic process $\X^K(n)$  generated by \eqref{density2d1} or, equivalently, 
by \eqref{Zprocess}.  As mentioned in Introduction, the resident population initially has a critical reproduction, 
and is at equilibrium, when a single mutant appears, so that $\X^K(0)=\big(\frac{[a_1 K]} K,\frac 1 K\big)\approx \x^{\mathrm{re}}=(a_1,0)$. 
Even though the probability of a mutant present at any time $n$ is positive, $\P(X_2^K(n)  >0)>0$, we do not say that it 
established itself until its numbers are proportional to its carrying capacity, in other words proportional to $K$. This can be formalized as
$$
\liminf_{K\to\infty} X_2^K(n) >0.
$$

For example, as we have seen above $\X^K(n)\to  \x^{\mathrm{re}}=(a_1,0)$ for any fixed $n$ as $K\to\infty$. 
This conveys that the mutant is not established by any fixed time $n$. We show however, that it may
establish itself at a time, which grows logarithmically with $K$.
More precisely, we prove that at time $[b\log K]$ with a certain constant $b$, 
\begin{equation}\label{limpos}
 \liminf_{K\to\infty} X_2^K\big([b\log K]\big) >0,\;\;(\mbox{and}\,\, \limsup_{K\to\infty} X_2^K\big([b\log K]\big)<\infty),
\end{equation}
whereas for $0< r < b$,
\begin{equation}\label{lim0}
\lim_{K\to\infty}\X^K([r\log K])\to  \x^{\mathrm{re}},
\end{equation}
and, therefore, 
$
X_2^K \big([r\log K]\big)  \xrightarrow[K\to \infty]{\P} 0,
$
in particular.

The logarithmic order of time of the mutant's establishment  can be roughly explained as follows. As the process starts near 
$\x^{\mathrm{re}}=(a_1,0)$, the state dependent splitting probabilities can be approximated, at least initially, by their values at 
$\x^{\mathrm{re}}$, giving probabilities of division $1/2$ and $ a_2 /( a_2+\gamma a_1)$ for the resident and the mutant populations 
respectively. Note that due to coexistence condition \eqref{coex},  the mutant process is supercritical with mean  
$$\rho = \frac{2a_2}{a_2+\gamma a_1} >1.$$ 
Hence it grows at the rate $\rho^n$, and it takes time 
$$
b\log K+O(1), \text{ with } b:=\frac{1}{\log\rho}
$$
for it to grow to the size 
proportional to $K$, as $K\to\infty$. In fact, this heuristics is correct, and made precise in the following result, which implies both 
\eqref{limpos} and \eqref{lim0}. We denote the fractional part of $x\in \Real_+$ by $\{x\}$.

\begin{theorem}\label{thm0} 
There exist a non-degenerate scalar random variable $W\ge 0$ and a function $\h(\x)$, whose entries are positive on 
the open half-plane  $\Real\times \Real_+$, such that
\begin{equation}\label{lim2}
\X^K \big([b\log K]\big) -  \h\left(\rho^{-\{\log_\rho K\}}\begin{pmatrix} 0 \\ W
\end{pmatrix}
\right)\xrightarrow[K\to \infty]{\P}  {\bf 0}.
\end{equation}
In particular,  along the subsequence of exact powers $K_j = \rho^j$, $j\in\mathbb{N}$,  
$$
\X^{K_j} \big(b\log K_j\big) \xrightarrow[j\to \infty]{\P}  \h\left(\begin{pmatrix} 0 \\ W
\end{pmatrix}
\right).
$$
 \end{theorem}
 
Let us now detail about the random variable $W$ and the function $\h(\cdot)$ appearing in this theorem. 
The approximate mutant process, mentioned in the heuristic explanation above, has the same splitting probability as 
the mutant component of $\mathbf{Z}(n)$ at $\x^{\mathrm{re}}$. 
More precisely, it is a supercritical Galton-Watson binary splitting, started with a single ancestor, $Y(0)=Z_2(0)=1$, and for $n\ge 1$
defined iteratively by
\begin{equation}\label{Yn}
Y(n+1) = \sum_{j=1}^{Y(n)} \zeta(n,j),
\end{equation} 
where the offsprings $\zeta(n,j)\in \{0,2\}$ are i.i.d. random variables with the constant splitting probability 
$\P\big(\zeta(n,j) = 2\big)=\rho/2$. 

It is well known that $W(n) =  \rho^{-n} Y(n)$ is a non-negative martingale. As such it converges almost surely to a limit,
$$W=\lim_{n\to\infty} W(n),$$
which is the random variable appearing in \eqref{lim2}.

The function $\h(\cdot)$ in Theorem \ref{thm0} is the limit of the $n$-fold iterated map $\f^n(\cdot)$ 
along the unstable manifold of  the dynamics in \eqref{map2d}.

\begin{theorem}\label{thmH}\ Under the basic assumptions stated,
the limit 
\begin{equation}\label{limH}
\h(\x)= \lim_{n\to\infty} \f^n(\x^{\mathrm{re}}+\x/\rho^n), \quad  \x\in \Real\times \Real_+ 
\end{equation}
exists, and the convergence is uniform on compacts.
\end{theorem}

\begin{figure}
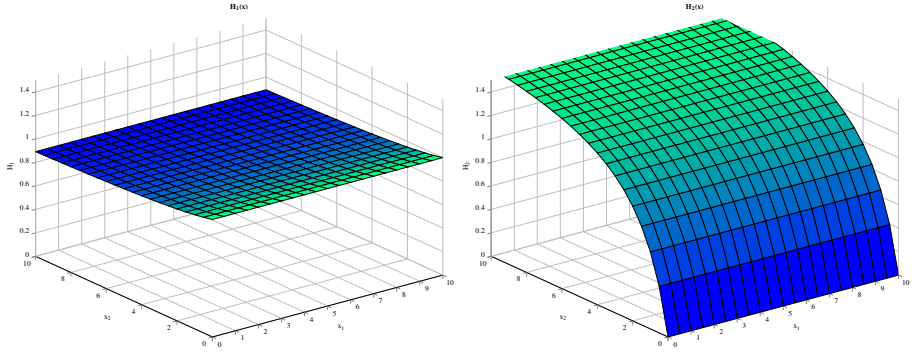
 
\input{H1versusx.tex}
\input{H2versusx.tex}
\caption{Numerical approximation of the function $\h(\x)$}
\label{fig1} 
\end{figure}

\begin{remark}   
\

\medskip 
\noindent
1) It is easy to see that $\h(\cdot)$ solves the Abel functional equation $\h(\x)=\f(\h(\x/\rho))$, subject to 
$\h(\bf 0)=x^{\mathrm{re}}$. While much is known of such equations in one dimension, in higher dimensions the 
theory is more involved.

\medskip
\noindent
2)
Numerical calculations indicate that $\h(\x)$ is constant with respect
to the resident population component $x_1$, see Figure \ref{fig1}. 
This is consistent with the criticality of that population at the density $a_1$: 
the global stability of the monomorphic dynamics makes those perturbations shrink to 0 when $\f$ is iterated.

\end{remark}

The next result describes the density process after establishment of the mutant, at times $[b\log K]+n$, $n= 1,  2, ...$; it shows that 
the population composition is governed by the deterministic nonlinear dynamics $\f^n$ started at the random point,
as illustrated on Figure \ref{fig0}.  Furthermore, it equally holds when $n$ is a negative integer.
Denote the random vector appearing in Theorem \ref{thm0} by
$$
\boldsymbol{\chi}(K) =\h\left(\rho^{-\{\log_\rho K\}}\begin{pmatrix} 0 \\ W
\end{pmatrix}
\right).
$$

\begin{corollary} \label{cor1}
For any $n\in \mathbb{Z}$,  
$$
\X^K([b\log K]+n)- \f^n (\boldsymbol{\chi}(K) ) \xrightarrow[K\to\infty]{\P}0.
$$
\end{corollary}

\begin{figure}
\begin{center}
\input{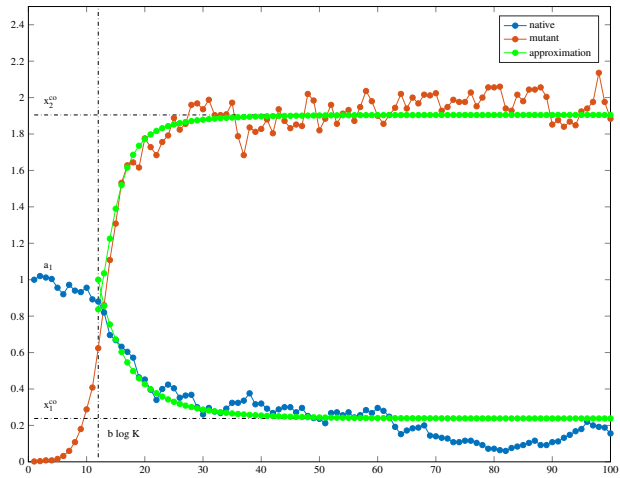}
\end{center}
\caption{ 
The density process versus the approximation with the random initial condition  
}
\label{fig0}
\end{figure}

The next corollary to Theorem \ref{thm0} answers the question what is the probability of the successful establishment of the mutant? Since the argument is short we present it here.
It is known that $\P(W=0)$ is exactly the extinction probability of the Galton-Watson process $Y(n)$. It is easily calculated to be $2/\rho-1$. But on the event $\{W=0\}$,
$\h\big((0,W)\big)=\h((0,0))=\x^{\mathrm{re}}$. 
Since on the complimentary event $W>0$, and $\h_2((0,w)) >0$ for $w>0$, we have the following corollary of \eqref{lim2}.
\begin{corollary} 
With probability $2(1-\rho^{-1})$ the mutant establishes itself alongside the large original population.  
 \end{corollary}

\section{Proofs}

\subsection{A preview}\label{ssec:3.1}
The proof is inspired by the observation that  supercritical populations, which start from a small number of individuals 
and develop on a habitat with a large but bounded capacity, grow initially as the Galton-Watson supercritical branching 
and then follow closely a deterministic curve, determined by the underlying nonlinear dynamics. This heuristics dates back at least 
to 50's, e.g., \cite{Kendall56} and \cite{Whittle55}, and the  already mentioned \cite{Metz}. A rigorous proof for epidemics is given in \cite{Mollison}, and in a wider context the rigorous implementations are relatively recent, see   
\cite{BHKK15}, \cite{BCK16} and \cite{CJK18}. 

Let us briefly sketch the ideas. The main ingredient is the Galton-Watson branching process $\Y(n)$,
whose components mimic the behaviour of those of $\Z(n)$ at the moment of mutant's appearance, that is, around the equilibrium 
point $\x^{\mathrm{re}}$. Thus its first, resident component $Y_1(n)$ is critical and starts at $[a_1K]$ and its second, mutant
component $Y_2(n)$ is supercritical with the offspring mean $\rho$, and it starts from a single individual. 
The two processes $\Z(n)$ and $\Y(n)$ are constructed on the same probability space and are coupled in such a way, 
that they remain close at least until time $n_c = [\log_\rho K^c]$ with some fixed constant $c\in (1/2,1)$.

Following the above heuristics the density $\X(n) =   \Z(n)/K =: \overbar \Z(n)$ is approximated by gluing the linearized 
stochastic process with the deterministic nonlinear dynamics,
$$
\widetilde \Z(n) = \begin{cases}
\overbar \Y(n) & n\le n_c \\
\f^{n-n_c}\big(\overbar \Y(n_c)\big) & n>n_c
\end{cases}
$$	
where $\overbar \Y(n) := \Y(n)/K$ is the density of the Galton-Watson branching. 
The assertion \eqref{lim2} of Theorem \ref{thm0} follows if we show that 

\medskip

\begin{enumerate}
\addtolength{\itemsep}{0.7\baselineskip}

\item the process $\widetilde \Z(n)$ does indeed approximate the target density $\overbar \Z(n)$
at time $n=[\log_\rho K]=\big[\frac 1 {\log\rho } \log K\big]=n_1$,  
\begin{equation}
\label{showme2}
\overbar{\Z}(n_1) - \widetilde \Z(n_1)  \xrightarrow[K\to\infty]{\P} 0,
\end{equation}

\item and the approximation $\widetilde \Z(n)$ behaves asymptotically as claimed in Theorem \ref{thm0},  
\begin{equation}
\label{showme1}
\widetilde \Z(n_1) -
\h\left(\rho^{-\{\log_\rho K\}}\begin{pmatrix}
0 \\ W
\end{pmatrix}
\right) \xrightarrow[K\to\infty]{\P}0.
\end{equation}

\end{enumerate}
 
\medskip

The main technical difficulty in proving \eqref{showme2} is to control the difference $\overbar{\Z}(n) - \widetilde \Z(n)$
on the time interval $[0,n_1]$, which itself grows with $K$. It turns out that the usual 
technique, based on straightforward linearization of the dynamics, does not provide bounds sharp enough in this case. 
Instead we construct a suitable coupling in Section \ref{sec:3.3}, which involves several additional auxiliary Galton-Watson processes. 

The key to the limiting expression in \eqref{showme1} is the representation 
\begin{equation}\label{eqeq}
\begin{aligned}
&
\f^{n_1-n_c} \big(\overbar \Y(n_c)\big) = \\
&
\f^{n_1-n_c} \Big(\x^{\mathrm{re}}+ \rho^{-\{\log_\rho K\}}\rho^{-(n_1-n_c)}\rho^{-n_c}\big( \Y(n_c)-K \x^{\mathrm{re}} \big)\Big).
\end{aligned}   
\end{equation}
It shows that the convergence in Theorem  \ref{thm0} follows once we prove the limit of Theorem \ref{thmH} 
and check that 
\begin{equation}\label{inview}
\rho^{-n_c}\big( \Y(n_c)-K \x^{\mathrm{re}} \big)\xrightarrow[K\to\infty]{\P}(0,W).
\end{equation}

The random variable $W$ is the martingale limit of the supercritical branching $Y_2(n)$, cf. \eqref{Yn}.
The most challenging element of the proof of this part is convergence \eqref{limH}, see Section \ref{sec:3.2} below.
Previously, it has been proved in \cite{CJK18} in dimension one, and analysis in higher dimensions, in our case two, 
requires a completely different approach. When convergence \eqref{lim2} is proved, the assertion of Corollary \ref{cor1} 
follows by continuity of $\f(\cdot)$.

The limit in equation \eqref{lim0} can be proved in a similar way: note that in this case, cf. \eqref{eqeq}
\begin{align*}
&
\f^{n_r-n_c} \big(\overbar \Y(n_c)\big) = \\
&
\f^{n_r-n_c} \Big(\x^{\mathrm{re}}+ \rho^{-\{\log_\rho K\}}\rho^{-(n_r-n_c)}\rho^{-(n_1-n_r)}\rho^{-n_c}\big( \Y(n_c)-K \x^{\mathrm{re}} \big)\Big),
\end{align*} 
where, in view of \eqref{inview},
$$
\rho^{-(n_1-n_r)}\rho^{-n_c}\big( \Y(n_c)-K \x^{\mathrm{re}} \big)\xrightarrow[K\to\infty]{\P} 0.
$$
We omit the proof of this part, which closely follows that of Theorem \ref{thm0} with obvious adjustments.

\subsection{The limit $\mathbf{H(x)}$} \label{sec:3.2}
 
In this section we construct limit \eqref{limH}, by means of a convergent telescopic series.

\subsubsection{An auxiliary recursion in dimension one} 
Let us start with an auxiliary one  dimensional quadratic recursion 
\begin{equation}\label{xrec}
x_{m,n} = \rho x_{m-1,n}\big(1+C x_{m-1,n}\big), \quad m = 1,...,n
\end{equation}
subject to initial condition $x_{0,n}=x/\rho^n$ with $x>0$, where $C\ge 0$ and $\rho>1$ are constant coefficients. 
In what follows we will need the following estimate on its solution.

\begin{lemma}\label{lem-main} There exists a finite function $\psi:\Real_+\mapsto \Real_+$, such that 
\begin{equation}\label{bnd}
x_{m,n} \le \psi(x) \rho^{m-n}, \quad m=1,...,n.
\end{equation}
\end{lemma}

\begin{proof}
Let us first note that no generality will be lost if $C=1$ is assumed. Indeed, if \eqref{xrec} is multiplied by $C$, 
the recursion 
$$
\widetilde x_{m,n} = p(\widetilde x_{m-1,n}), \quad m=1,...,n
$$
with $p(x)= \rho x(1+x)$ is obtained for the rescaled sequence $\widetilde x_{m,n} = C x_{m,n}$.
Hence if \eqref{bnd} holds for $\widetilde x_{m,n}$ with some $\widetilde \psi(x)$, then it holds for $x_{m,n}$ 
with $\psi(x) := C^{-1} \widetilde \psi(Cx)$. From here on we set $C=1$.

Since  
$
x_{m,n} \ge \rho x_{m-1,n}, 
$
proving the desired bound amounts to showing  
\begin{equation}\label{sup}
\sup_n x_{n,n}<\infty, \quad \forall x\ge 0.
\end{equation}
To this end, consider the Schr\"oder functional equation 
\begin{equation}\label{Schroder}
\phi(f(x))=s\phi(x), \quad x\in [0,\infty)
\end{equation}
where $s=:1/\rho \in(0,1)$ and $f(x)=\sqrt{\frac 1 4 + sx}-\frac 1 2$ is the inverse of the parabola $p(\cdot)$ on $\Real_+$. 
The function $f(x)$ satisfies the following conditions 

\begin{enumerate}

\item $f$ is continuous and strictly increasing on $[0,\infty)$

\item $f(0)=0$ and $0<f(x)<x$ for $0<x<\infty$

\item $f(x)/x\to s$ as $x\to 0+$

\item $f(x)$ is concave (and therefore $f(x)/x$ is decreasing on $\Real_+$)

\item for all $\delta >0$   
$$
\int_0^\delta \frac{|f(x)-sx|}{x^2}dx<\infty.
$$

\end{enumerate}

Under these conditions \cite{S69} shows that the limit 
$$
\phi(x):= \lim_n \frac{f^n(x)}{s^n}, \quad x\in [0,\infty)
$$
exists, solves \eqref{Schroder} and satisfies the following properties 

\begin{enumerate}
\item[a)] $0<\phi(x)<\infty$ on $(0,\infty)$ (nontrivial limit), $\phi(0+)=0$

\item[b)]  $\phi(x)/x$ is monotone on $(0,\infty)$

\item[c)] 
$
\phi'(0+)=1
$

\item[d)] $\phi(x)$ is invertible\footnote{see the remark in the paragraph following (3.1) in \cite{S69}}

\end{enumerate}
Changing the variable  in \eqref{Schroder} to  $y=f(x)$, we get
$$
\phi(y)=s\phi(p(y)), \quad y\in \Real_+,
$$
and, inverting, we obtain the conjugacy  
$$
p(y) = \phi^{-1}\big(\rho \phi(y)\big).
$$ 
Hence 
$$
x_{n,n} = p^n(x/\rho^n) = \phi^{-1}\big(\rho^n \phi(x/\rho^n)\big)\xrightarrow[n\to\infty]{}  \phi^{-1}\big(x \phi'(0+)\big) = \phi^{-1}(x).
$$
In particular, \eqref{sup} and, therefore also \eqref{bnd}, hold. \qed
\end{proof}

\subsubsection{Properties of $\f(\cdot)$.}
Let us summarize some relevant properties of the function $\f(\cdot)$, which governs the deterministic dynamics in \eqref{map2d}. 
Since $f_1(\x)-x_1$ and $f_2(\x)-x_2$ change signs across the lines $x_1+\gamma x_2 =a_1$ and $\gamma x_1 + x_2 =a_2$ respectively, as shown at the phase portrait Figure \ref{fig3}, the coexistence equilibrium $\x^{\mathrm{co}}$ is globally stable. The local behaviour around the unstable fixed point $\x^{\mathrm{re}}=(a_1,0)$ is determined 
the Jacobian of $\f(\cdot)$ at $\x^{\mathrm{re}}$,
\begin{equation}\label{Adef}
A:= D   \f( \x^{\mathrm{re}})
=\begin{pmatrix}
\frac 1 2 & -\frac{ \gamma }{2}\\
0 & \phantom{+}\rho 
\end{pmatrix}, \ \text{where}\;\;\rho = m_{2}(\x^{\mathrm{re}})=\frac{2 a_2 }{ a_2+\gamma a_1 }.
\end{equation}
%

\begin{figure}[t]
\begin{center}
\input{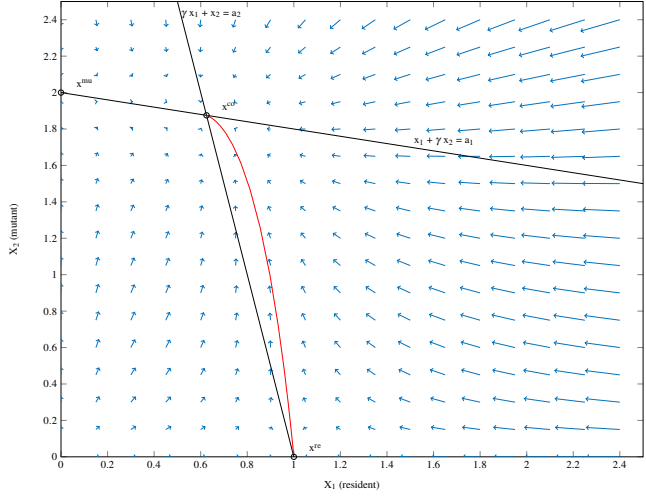}
\end{center}
\caption{
The phase-portrait of the deterministic system \eqref{map2d}; the trajectory from a small vicinity of the resident equilibrium
$\x^{\mathrm{re}}$ to the coexistence equilibrium $\x^{\mathrm{co}}$ is depicted in red.
}
\label{fig3} 
\end{figure}

To study perturbations around $\x^{\mathrm{re}}$ it will be convenient to consider the translation 
\begin{equation}\label{eqng}
\g(\x):= \f(\x^{\mathrm{re}}+\x)-\x^{\mathrm{re}},
\end{equation}
with $\g(\mathbf{0})=\mathbf{0}$ and Jacobian $D\,\g(\mathbf{0})=A$. 
In particular, existence of the limit \eqref{limH} follows from that of $\lim_{n\to\infty} \g^n(\x)$.

Note that for $x_1\in [x^{\mathrm{co}}_1, \infty)$ and $x_2 \in [0,x^{\mathrm{co}}_2]$, formulas \eqref{dynf} for the entries of the function 
$\f(\cdot)$ and the configuration of the fixed points \eqref{stableFP} imply 
$$
f_1( \x ) =   \frac{2x_1a_1}{a_1+x_1+\gamma x_2} \ge  \frac{2x^{\mathrm{co}}_1a_1}{a_1+x^{\mathrm{co}}_1+\gamma x^{\mathrm{co}}_2}=x^{\mathrm{co}}_1 
$$
and 
$$
f_2(\x)  = \frac{2x_2a_2}{a_2+\gamma x_1+ x_2} \le \frac{2x_2^{\mathrm{co}}a_2}{a_2+\gamma x_1^{\mathrm{co}}+ x_2^{\mathrm{co}}}=x_2^{\mathrm{co}}.
$$
Hence the subset $\widetilde E :=  [x^{\mathrm{co}}_1, \infty)\times [0,x^{\mathrm{co}}_2]\subset\Real^2_+$ is forward invariant under $\f(\cdot)$, namely, 
$\f(\widetilde E)\subseteq \widetilde E$.
Then by \eqref{eqng} the subset 
$$
E = [x^{\mathrm{co}}_1-a_1, \infty)\times [0,x^{\mathrm{co}}_2] \subset \Real\times \Real_+
$$ 
is forward invariant under $\g(\cdot)$. 
 
In what follows, $\|\cdot \|$ stands for the $\ell_\infty$ norm for vectors and 
the corresponding operator norm for matrices. 
In particular, the matrix $A$ defined in \eqref{Adef} satisfies $\|A\|= \rho>1$.
The linear subspace $E_0 =\{\x\in \Real^2: x_2=0\}$ is invariant under $A$ and 
\begin{equation}\label{normA}
\sup_{\x\in E_0} \frac{\|A\x\|}{\|\x\|}=\frac 1 2.
\end{equation}
Below $C$, $C_1$, etc. stand for constants, which depend  only on $a_1$, $a_2$ and $\gamma$
and whose values may change from line to line. 

The first coordinate of $\g(\cdot)$ can be written as   
\begin{align*}
g_1(\x)  = & \frac{2 a_1 (a_1+x_1) }{2a_1 +x_1+\gamma x_2}-a_1 = \frac{   a_1x_1   }{2a_1 +x_1+\gamma x_2}
-
\frac{  a_1\gamma x_2}{2a_1 +x_1+\gamma x_2} =\\
&
\frac 1 2 x_1 \left( 
1
- \frac{   x_1+\gamma x_2   }{2a_1 +x_1+\gamma x_2} \right)
-\frac{\gamma} 2 x_2\left(1-
\frac{  x_1+\gamma x_2}{2a_1 +x_1+\gamma x_2}
\right),
\end{align*}
and, similarly,  
\begin{align*}
g_2(\x) =  \rho x_2\left(1-    \frac{      x_2 +\gamma x_1}{a_2 +\gamma  a_1+  x_2+\gamma x_1 }\right). 
\end{align*}

Hence $\g(\x)$ has the form 
\begin{equation}\label{grec}
\g(\x) = \big(I-B(\x)\big)A \x
\end{equation}
where matrix $B(\x)$ satisfies the bound
\begin{equation}\label{Bbnd}
\|B(\x)\|\le C \|\x\|, \quad \x\in E
\end{equation}
with a constant $C$.
Similar calculation also shows that for $\x,\y\in E$
\begin{equation}\label{gxgy}
g(\x)-g(\y) = \big(A+F(\x,\y)\big)(\x-\y)
\end{equation}
where matrix $F(\x,\y)$ satisfies 
\begin{equation}\label{Fbnd}
\|F(\x,\y)\| \le C \big(\|\x\| \vee \|\y\|\big).
\end{equation}

These formulas and the bound from Lemma \ref{lem-main} give the following growth estimate.

\begin{lemma}
For  any $\x\in  \Real\times \Real_+$ and all $n$ large enough, 
\begin{equation}\label{gbnd}
\big\|\g^m(\x/\rho^n)\big\| \le \psi(\|\x\|) \rho^{m-n}, \quad m=1,...,n
\end{equation}
with a finite function $\psi(r)$, $r\ge 0$.
\end{lemma}

\begin{proof}
For any $\x\in \Real\times \Real_+$ and all $n$ large enough $\x/\rho^n\in E$ and, since $E$ is invariant, $\g^m(\x/\rho^n)\in E$
for all $m$. Hence by \eqref{grec}, the sequence $\x_{m,n}=\g^m(\x/\rho^n)$ satisfies 
\begin{align*}
\|\x_{m+1,n}\| = & \| \g\big(\x_{m,n}\big) \|  = \big\|
\big(I-B(\x_{m,n})\big)A \x_{m,n}
\big\| \le \\
& \| I-B(\x_{m,n})\|\|A \|\|\x_{m,n}\| \le \rho\|\x_{m,n}\| (1+C \|\x_{m,n}\|).
\end{align*}
By induction $\|\x_{m,n}\|\le x_{m,n}$, where $x_{m,n}$  solves \eqref{xrec} subject to $x_{0,n}=\|\x\|/\rho^n$, 
and  the claim  follows from  Lemma \ref{lem-main}.\qed
\end{proof}

\subsubsection{Proof of Theorem \ref{thmH} } 

We will argue that the increments of $\g^n(\x/\rho^n)$ are absolutely summable, uniformly over compacts in $\Real\times \Real_+$. 
Let $n$ be large enough so that $\x/\rho^n\in E$ and therefore, by invariance, $\g^{m}(\x/\rho^n)\in E$ for all $m\ge 1$.
Consider the array 
$$
\g^{m}(\x/\rho^{n+1}) - \g^{m-1}(\x/\rho^n), \quad m=1,...,n.
$$
For $m=1$, due to  \eqref{grec}, 
\begin{align*}
\g (\x/\rho^{n+1})- \x/\rho^n =\, & A \x/\rho^{n+1}- \x/\rho^n-B(\x/\rho^{n+1})A \x/\rho^{n+1} = \\
&
\rho^{-n}\big( A/\rho  -I\big)\x    + \rho^{-2n} \mathbf{v}_n =: \rho^{-n} \mathbf{u}  + \rho^{-2n} \mathbf{v}_n,
\end{align*}
where $\mathbf{u}\in E_0$ and, in view of \eqref{Bbnd}, $\mathbf{v}_n$ is a sequence of vectors, whose norm is bounded 
uniformly in $n$.
Both vectors $\mathbf{u}$ and $\mathbf{v}_n$ depend continuously on $\x$, which is omitted from the notations.
For $m\ge 1$,  \eqref{gxgy} implies 
\begin{align*}
& 
\g^{m+1}(\x/\rho^{n+1}) - \g^{m}(\x/\rho^n) = \g\circ \g^{m} (\x/\rho^{n+1}) - \g\circ \g^{m-1}(\x/\rho^n) = \\
&
\Big(A + F\big(\g^m (\x/\rho^{n+1}),\g^{m-1}(\x/\rho^n)\big)\Big)\Big(\g^m (\x/\rho^{n+1})-\g^{m-1}(\x/\rho^n)\Big)
\end{align*}
and, letting $
F_{m,n} := F\big(\g^m (\x/\rho^{n+1}),\g^{m-1}(\x/\rho^n)\big) 
$
,  we get
\begin{equation}\label{D1}
\g^{n+1}(\x/\rho^{n+1}) - \g^n(\x/\rho^n) = 
\left\{\prod_{m=1}^n
\Big(A + F_{m,n}\Big)
\right\}
\big(\rho^{-n} \mathbf{u} + \rho^{-2n} \mathbf{v}_{n}\big).
\end{equation}
 
Since  $\|A\|=\rho$, by virtue of \eqref{gbnd} and \eqref{Fbnd}
\begin{equation}\label{2ndterm}
\begin{aligned}
&
\left\|
\left\{\prod_{m=1}^n
\big(A + F_{m,n}\big)
\right\}
 \rho^{-2n} \mathbf{v}_{n} 
\right\| \le 
\rho^{-2n}\|\mathbf{v}_{n}\| \prod_{m=1}^n \Big(\big\|A \big\|+ \big\|F_{m,n} \big\|\Big) \le  \\
&
\rho^{-2n} C_1 \prod_{m=1}^n \Big(\rho+ C  \big(\|\g^m (\x/\rho^{n+1})\|\vee \|\g^{m-1}(\x/\rho^n)  \|\big)\Big)\le \\
&
\rho^{-2n}C_1 \prod_{m=1}^n \big(\rho+   C_2 \rho^{m-n}\big)= 
\rho^{- n}C_1 \prod_{m=1}^n \big(1+   C_2 \rho^{m-n-1}\big)\le C_3\rho^{- n},
\end{aligned}
\end{equation}
where constants $C_j$'s depend continuously on $\|x\|$.

Let us now bound the remaining term in \eqref{D1}. To this end, observe that
\begin{align*}
\prod_{m=1}^n\Big(A + F_{m,n}\Big) = & \prod_{m=2}^n\Big(A + F_{m,n}\Big) F_{1,n} + \\
&
\prod_{m=3}^n\Big(A + F_{m,n}\Big) F_{2,n} A + \\
& 
\prod_{m=4}^n\Big(A + F_{m,n}\Big) F_{3,n} A^2 + \cdots +\\
& \Big(A + F_{n,n}\Big)F_{n-1,n}A^{n-2} + F_{n,n} A^{n-1} + A^n.
\end{align*}
Since $\mathbf{u}\in E_0$ and $E_0$ is invariant under $A$, we have $\|A^k\mathbf{u}\|\le (1/2)^k\|\mathbf{u}\|$ due to
\eqref{normA}.  Therefore for all $k=0,...,n-2$
\begin{align*}
&
\Big\|\prod_{m=k+2}^n\Big(A + F_{m,n}\Big) F_{k+1,n} A^{k} \mathbf{u} \Big\|
\le \\
&
(1/2)^k\|\mathbf{u}\|  \Big\|\prod_{m=k+2}^n\Big(A + F_{m,n}\Big) F_{k+1,n}   \Big\|\le \\
&
C_1 (1/2)^k \|F_{k+1,n}\|    \prod_{m=k+2}^n\Big(\|A\| + \|F_{m,n}\|\Big) \le \\
&
 C_2 (1/2)^k\|\x\|\rho^{k-n}    \prod_{m=k+2}^n\Big(\rho +  C  \|\x\|\rho^{m-n}\Big) \le \\
&
 C_3 (1/2)^k      \prod_{m=k+2}^n\Big(1 +  C_4\rho^{m-n}\Big)\le \\
&
 C_3    (1/2)^k \exp \Big(  C_4\sum_{m=k+2}^n \rho^{m-n} \Big)\le C_5 (1/2)^k ,
\end{align*}
where all $C_j$'s depend continuously on $\x$.
Consequently
$$
\left\|\left\{
\prod_{m=1}^n
\Big(A + F_{m,n}\Big)
\right\}
\rho^{-n} \mathbf{u}   
\right\|\le  C_6  \rho^{-n}.
$$
Plugging this and \eqref{2ndterm} into \eqref{D1} yields 
$$
\big\|\g^{n+1}(\x/\rho^{n+1}) - \g^n(\x/\rho^n)\big\| \le C_7  \rho^{-n},
$$ 
where $C_7$ depends continuously on $\x$. This implies that $\g^n(\x/\rho^n)$ converges as $n\to\infty$, uniformly on compacts. 
Existence of the limit $\mathbf{H}(\x)$ in \eqref{limH} now follows from \eqref{eqng}, and Theorem \ref{thmH} is proved.\qed

\subsection{The main approximation}\label{sec:3.3}
In this section we construct the random variable $W$ and prove  
convergence \eqref{lim2}. To this end, let $U(n,j)$ and $V(n,j)$ be i.i.d.
random variables distributed uniformly over the unit interval $[0,1]$ and define $\xi_1(n,j)$ and $\xi_2(n,j)$ in \eqref{Zprocess} 
as 
\begin{equation}\label{UVrv}
\begin{aligned}
\xi_1(n,j)=\, & 2\cdot {\mathbf 1}\bigg\{U(n,j) \le \frac{a_1K}{a_1K+Z_1(n)+\gamma Z_2(n)}\bigg\}, \\
\xi_2(n,j)  =\,  & 2\cdot {\mathbf 1}\bigg\{V(n,j) \le \frac{a_2K}{a_2K+\gamma Z_1(n)+  Z_2(n)}\bigg\}.
\end{aligned}
\end{equation}
Define Galton-Watson branching process $\Y(n)$ with components  
\begin{align*}
Y_1(n+1)   =& \sum_{j=1}^{Y_1(n) } 2\cdot {\mathbf 1}\bigg\{U(n,j) \le \frac 1 2\bigg\}, \\ 
Y_2(n+1)   =& \sum_{j=1}^{Y_2(n)} 2\cdot {\mathbf 1}\bigg\{V(n,j) \le \frac 12 \rho  \bigg\},
\end{align*}
subject to $Y_1(0) = [a_1 K]$ and $Y_2(0) = 1$, and the corresponding density process $\overbar\Y(n)=\Y(n)/K$. 
Note that $Y_2(n)$ coincides in distribution with the process defined in \eqref{Yn}
and 
\begin{equation}
\label{Y2n}
\rho^{-n} Y_2(n)\xrightarrow[n\to\infty]{\P}W.
\end{equation}
Finally, for a fixed constant $c\in (\frac 1 2,1]$ define    
$$
n_c(K) := [\log_\rho K^c] = \left[\frac{c}{\log \rho}\log K\right].
$$ 
In particular, $n_1(K) = [b\log K] = \big[ \log_\rho K\big]$, cf. Theorem \ref{thm0}.  

\medskip As explained in Section \ref{ssec:3.1}, the limit \eqref{lim2} follows from  \eqref{showme2} and \eqref{showme1}.

\subsubsection{Proof of \eqref{showme1}}

Since 
$
\E  Y_1(n)=[a_1 K]  
$
and 
$
\var  \big(Y_1(n)\big) \le n a_1K
$ we have
\begin{equation} 
 \E\big(  Y_1(n_c)-a_1 K \big)^2 \le   a_1 K  \log_\rho K^c
\end{equation}
and hence for any $c\in (\frac 1 2,1)$, 
\begin{equation}\label{K1c}
\rho^{-n_c} \big(    Y_1(n_c)-  a_1 K \big)\xrightarrow[K\to\infty]{\P} 0.
\end{equation}
This along with \eqref{Y2n} implies \eqref{inview}, and in view of  
representation \eqref{eqeq}, the limit in \eqref{showme1} follows  
by the continuous mapping theorem and the uniform convergence in \eqref{limH}.\qed

\subsubsection{Proof of \eqref{showme2}.}
Since 
$$
\big\|\overbar{\Z}(n_1) - \widetilde \Z(n_1)\big\|  \le\,  
\big\|\overbar{\Z}(n_1)-\f^{n_1-n_c} (\overbar{\Z}(n_c))\big\|+ 
\big\|\f^{n_1-n_c} (\overbar{\Z}(n_c))- \f^{n_1-n_c}(\overbar{\Y}(n_c))\big\|
$$
it suffices to prove that
\begin{equation}\label{showme21}
\big\|\overbar{\Z}(n_1)-\f^{n_1-n_c} \big(\overbar{\Z}(n_c)\big)\big\|\xrightarrow[K\to\infty]{\P} 0
\end{equation}
and 
\begin{equation}\label{showme22}
\big\|\f^{n_1-n_c} \big(\overbar{\Z}(n_c)\big)- \f^{n_1-n_c}\big(\overbar \Y(n_c)\big)\big\|\xrightarrow[K\to\infty]{\P} 0.
\end{equation}

Let us first prove \eqref{showme21}.
Recall that the density process $\X(n)=\overbar{\Z}(n)$ solves  \eqref{density2d1},
$$
\overbar{\Z}(n) = \f\big(\overbar{\Z}(n-1)\big) +\frac 1 {\sqrt K} \boldsymbol{\eta}(n),
$$
and hence the difference $\bdelta(n) := \overbar{\Z}(n)-\f^{n-n_c} (\overbar{\Z}(n_c))$  satisfies 
$$
\bdelta(n)  = \f\big(\overbar{\Z}(n-1)\big) -\f\big(\overbar{\Z}(n-1)-\bdelta(n-1)\big) +
\frac 1 {\sqrt K} \boldsymbol{\eta}(n), \quad n> n_c
$$
subject to $\bdelta(n_c)=0$. A direct calculation shows that the Jacobian of $\f(\cdot)$ is bounded
$$
\widetilde \rho := \sup_{\x\in \Real^2_+}\|D\, \f(\x)\|_\infty\in (\rho, 2],
$$ 
Hence $\f(\cdot)$ is $\widetilde \rho$-Lipschitz on $\Real_+^2$ with respect to $\ell_\infty$ norm and  
$$
\|\bdelta(n)\| \le \widetilde\rho\, \|\bdelta(n-1)\|+ \frac 1 {\sqrt{K}} \|\boldsymbol{\eta}(n)\|.
$$
Let $\beta:= \log_\rho \widetilde \rho > 1$, then 
\begin{align*}
\E \|\bdelta(n_1)\| & \le 
\frac 1{\sqrt K} \sum_{j=n_c}^{n_1} \widetilde {\rho}^{\, n_1-j} \E \|\boldsymbol{\eta}(j)\|\le   \\
&
\frac 1{\sqrt K}    (n_1-n_c)\widetilde\rho^{n_1-n_c} \sup_{j\le n_1}\E \|\boldsymbol{\eta}(j)\|\le 
C   K^{(1-c)\beta -\frac 1 2}\log_{\rho}K^{1-c} \to 0,
\end{align*}
where  convergence holds if $c$ is chosen close enough to $1$. 
This proves \eqref{showme21}. 

To check \eqref{showme22},
write 
\begin{align*}
&
\Big\|\f^{n_1-n_c} \big(\overbar{\Z}(n_c)\big)- \f^{n_1-n_c}\big(\overbar{\Y}(n_c)\big)\Big\|=  \\
&
\Big\|   \f^{n_1-n_c}  \Big(\x^{\mathrm{re}}+\rho^{-\{\log_\rho K\}}\rho^{-(n_1-n_c)}  \rho^{-n_c}(\Z(n_c)-K\x^{\mathrm{re}}) \Big) \\
&
 -\f^{n_1-n_c} \Big(\x^{\mathrm{re}}+\rho^{-\{\log_\rho K\}}\rho^{-(n_1-n_c)}  \rho^{-n_c}(\Y(n_c)-K\x^{\mathrm{re}}) \Big)
\Big\|.
\end{align*}
Since, by \eqref{Y2n} and \eqref{K1c}, the sequence $\rho^{-n_c}(\Y(n_c)-K\x^{\mathrm{re}})$ converges  
to $(0,W)$ in probability and, by Theorem \ref{thmH}, the sequence $\f^n(\x^{\mathrm{re}}+\x/\rho^n)$ converges uniformly on compacts to $\h(\x)$, it suffices to show that 
$$
\rho^{-n_c} \|\Z(n_c)-\Y(n_c)\|\xrightarrow[K\to\infty]{\P}0,
$$
that is,
\begin{equation}
\label{checkme}
K^{-c}\big|Z_j(n_c)-Y_j(n_c)\big| \to 0,\quad j=1,2,
\end{equation} 
where $c\in (\frac 1 2, 1)$ has been already fixed in the previous calculations. 
We prove \eqref{checkme} for $j=2$, omitting the similar proof for the case $j=1$.

To this end, choose arbitrary constants  
$$
\alpha_{1\ell}, \alpha_{1u}, \alpha_2 \in (c,1)\subset (\tfrac 1 2 ,1), 
$$
and, using the same random variables  $U(n,j)$ and $V(n,j)$ as in \eqref{UVrv}, define two 
additional auxiliary Galton-Watson branching processes $\L(n)$ and $\U(n)$ with the entries 
\begin{align*}
& 
L_1(n) = \sum_{j=1}^{L_1(n-1)}2\cdot {\mathbf 1}\bigg\{U(n,j) \le  \frac 1 2 r^-_K  \bigg\}, \quad L_1(0)=[a_1K], \\
&
U_1(n) = \sum_{j=1}^{U_1(n-1)}2\cdot {\mathbf 1}\bigg\{U(n,j) \le  \frac 1 2 r^+_K  \bigg\}, \quad U_1(0)=[a_1K],
\end{align*}
where  
$$
r^-_K := \frac{2a_1}{a_1+ a_1(1+ K^{\alpha_{1u}-1})+  \gamma K^{\alpha_2-1}} < 1
\quad
\text{and}
\quad 
r^+_K := \frac {2a_1} {a_1+ a_1(1- K^{\alpha_{1\ell}-1}) } > 1,
$$
and 
\begin{align*}
& 
L_2(n) = \sum_{j=1}^{L_2(n-1)}2\cdot {\mathbf 1}\bigg\{V(n,j) \le \frac 12 \rho^-_K  \bigg\}, \quad L_2(0)=1, \\
&
U_2(n)  = \sum_{j=1}^{U_2(n-1)}2\cdot {\mathbf 1}\bigg\{V(n,j) \le \frac 12 \rho^+_K  \bigg\}, \quad U_2(0) =1,
\end{align*}
where 
$$
\rho^-_K = \frac{2 a_2}{a_2+\gamma a_1 (1+K^{\alpha_{1u}-1})  +K^{\alpha_2-1}}< \rho
\quad
\text{and}
\quad
\rho^+_K = \frac{ 2a_2}{a_2+\gamma a_1 (1-K^{\alpha_{1\ell}-1})}>\rho.
$$
 
Define the exit times
\begin{equation}\label{taus}
\begin{aligned}
\tau^{1\ell}  
&
= \min\big\{n:  Z_1(n)  \le  a_1 (K-K^{\alpha_{1\ell} }) \big\},
\\
\tau^{1u}  
&
= \min\big\{n:  Z_1(n) \ge  a_1 (K+K^{\alpha_{1u} }) \big\}, 
\\
\tau^2  
&
= \min\big\{n: Z_2(n)\ge K^{\alpha_2 }  \big\}.
\end{aligned}
\end{equation}
The random variable 
$
\tau = \tau^{1\ell} \wedge \tau^{1u} \wedge \tau^2 
$ 
is a coupling time for the above processes, in the sense that
$$
\P\big(L_2(n) \le Y_2(n)\le U_2(n)\big)=1, 
$$ 
and 
$$
\{\tau\ge n\}\subseteq \big\{L_2(n)   \le Z_2(n)\le U_2(n)\big\}.
$$
Hence  
$$
\big|Z_2(n) - Y_2(n)\big|\le 
 \big(U_2(n) - L_2(n)\big)\one{\tau\ge n}
+ \big|Z_2(n) - Y_2(n)\big|\one{\tau< n}.
$$
Convergence \eqref{checkme} for $j=2$ holds, if we show that 
\begin{equation}\label{raz}
K^{-c}\big(U_2(n_c) - L_2(n_c)\big)\xrightarrow[n\to\infty]{\P} 0,
\end{equation}
and 
\begin{equation}\label{PK}
 \P(\tau < n_c)\xrightarrow[K\to\infty]{}0,
\end{equation}
since $\big\{K^{-c}\big|Z_2(n) - Y_2(n)\big|\one{\tau<n}\ge \eps\big\}\subseteq \{\tau < n\}$ for any $\eps>0$.

\

The limit in \eqref{raz} holds, since 
\begin{align*}
&
K^{-c} \E \big(U_2(n_c)- L_2(n_c)\big) \le   
  (\rho^+_K/\rho)^{n_c}-(\rho^-_K/\rho)^{n_c}  \le  \frac{\rho^+_K-\rho^-_K}{\rho} \big(\rho^+_K/\rho\big)^{n_c}n_c
\le \\
& C\,     \Big|K^{\alpha_{1\ell}-1}+K^{\alpha_{1u}-1}+K^{\alpha_2-1}\Big|   
\big(1+ O(K^{\alpha_{1\ell}-1})\big)^{\log_\rho K^c}\log_\rho K^c \xrightarrow[K\to\infty]{}0.
\end{align*}

It is left to prove \eqref{PK}. Since  
$$
\P(\tau < n_c) \le \P(\tau^{1\ell} <n_c) + \P(\tau^{1u} <n_c) + \P(\tau^2 <n_c),
$$
it suffices to check that each of the exit times, c.f. \eqref{taus}, 
\begin{align*}
\tau^{1\ell}(\alpha)  & = \min\big\{n:  Z_1(n)  \le  a_1 (K-K^{\alpha }) \big\},\\
\tau^{1u}(\alpha) & = \min\big\{n:  Z_1(n) \ge  a_1 (K+K^{\alpha }) \big\}, 
\\
\tau^2 (\alpha) & = \min\big\{n: Z_2(n)\ge K^{\alpha }  \big\},
\end{align*}
satisfy  
\begin{subequations}
\begin{align}
\label{pr3}
&
\P\big(\tau^{1\ell}(\alpha)\le n_c \big)\xrightarrow[K\to\infty]{} 0,  \\
&
\label{pr2}
\P\big(\tau^{1u}(\alpha)\le n_c \big)\xrightarrow[K\to\infty]{} 0, \\
&
\label{pr1}
\P\big(\tau^{2}(\alpha)\le n_c \big)\xrightarrow[K\to\infty]{} 0, 
\end{align}
\end{subequations}
for {\em any} $\alpha\in ( c , 1)$. 
To this end, we can reuse the auxiliary processes $\L(n)$ and $\U(n)$, defined above, with appropriately chosen 
parameters $\alpha_{1\ell}$, $\alpha_{1u}$ and $\alpha_2$.
Define exit times
\begin{align*}
\sigma^2 &= \min\{n: U_2(n)\ge K^{\alpha_2}  \}, \\
\sigma^{1u} &= \min\{n: U_1(n) \ge a_1(K+K^{\alpha_{1u}})\}, \\
\sigma^{1\ell} &= \min\{n: L_1(n)   \le a_1(K-K^{\alpha_{1\ell}})\}.
\end{align*}

To prove \eqref{pr1}, we can choose $\alpha_{1\ell} = \alpha_{1u}=\alpha_2 = \alpha$.
By construction, 
$$
\{\tau^2(\alpha) <n_c\}\subseteq \{\sigma^2< n_c\},$$
and therefore 
\begin{align*}
&
\P(\tau^2(\alpha) < n_c)\le \P(\sigma^2<n_c) =  \P\left(\max_{n\le n_c}U_2(n)\ge K^{\alpha_2}\right) = \\
&
\P\left((\rho_K^+)^{-n_c}\max_{n\le n_c}U_2(n)\ge(\rho_K^+)^{-n_c}K^{\alpha_2}\right)\le \\
&
\P\left(\max_{n\le n_c}(\rho_K^+)^{-n} U_2(n)\ge(\rho_K^+)^{-n_c}K^{\alpha_2}\right)\stackrel{\dagger}{\le} 
(\rho_K^+)^{n_c}K^{-\alpha_2} \le C K^{c-\alpha_2}\xrightarrow[K\to\infty]{} 0,
\end{align*} 
where $\dagger$ holds by Doob's inequality \cite[Theorem VII.3.3, p. 493, eq. (11)]{Shbook}, applied to the martingale $(\rho_K^+)^{-n} U_2(n)$.

To prove \eqref{pr2}, let us choose $c < \alpha_{1u}=\alpha_2 < \alpha_{1\ell}=\alpha$. 
By construction, 
$$
\{\tau^{1u}(\alpha) < n_c\}\subseteq \{\sigma^{1u}<n_c\},
$$ and since 
the process $[a_1K]-L_1(n)$ is a submartingale, 
\begin{align*}
&
\P(\tau^{1\ell}(\alpha)<n_c)\le \P(\sigma^{1\ell}<n_c) =  \P\left(\min_{n\le n_c}L_1(n)\le a_1 (K-K^{\alpha_{1\ell}})\right) = \\
& 
\P\left(\max_{n\le n_c}\big([a_1K] -L_1(n)\big)\ge a_1 K^{\alpha_{1\ell}} -1\right)\stackrel{\dagger}{\le} C K^{-\alpha_{1\ell}}\E \big|[a_1K]-L_1(n_c)\big|
\stackrel{\ddagger}{\le} \\
&
C K^{-\alpha_{1\ell}} \big| [a_1K] -\E L_1(n_c)\big|
+
C K^{-\alpha_{1\ell}}\sqrt{\var \big( L_1(n_c)\big)},
\end{align*}
where $\dagger$ is another variant of Doob's inequality \cite[Theorem VII.3.1, p. 492, eq. (1)]{Shbook}  and $\ddagger$ 
holds by the Jensen inequality. 
The first term  satisfies
\begin{align*}
&
  K^{-\alpha_{1\ell}} \big|[a_1K]-\E L_1(n_c)\big|
= 
 K^{-\alpha_{1\ell}} [a_1K] \big|1-  (r^-_K)^{n_c}\big|\le \\
&
  C_1 K^{1-\alpha_{1\ell}}    \Big|1-   \Big(1-  K^{\alpha_{1u}-1}-  \frac{\gamma}{a_1}K^{\alpha_2-1}\Big)^{n_c}\Big| \le\\
& C_2  K^{1-\alpha_{1\ell}} \big(K^{\alpha_{1u}-1}+K^{\alpha_2-1}\big) n_c  \le C_3 
  \big(K^{\alpha_{1u}-\alpha_{1\ell}} + K^{\alpha_2-\alpha_{1\ell}} \big) \log_\rho  K \xrightarrow[K\to\infty]{}0,
\end{align*}
where the convergence holds by the choice  $\alpha_{1u}=\alpha_2<\alpha_{1\ell}$. 
The second term satisfies
$$
K^{-\alpha_{1\ell}}\sqrt{\var ( L_1(n_c))} \le K^{-\alpha_{1\ell}}\sqrt{ a_1 K n_c (r^-_K)^{n_c}}\le C K^{\frac 1 2-\alpha_{1\ell}}\log_\rho  K 
\xrightarrow[K\to\infty]{}0.
$$
Finally, to prove \eqref{pr3}, let us choose $c< \alpha_2 = \alpha_{1\ell} < \alpha_{1u}=\alpha$, then 
\begin{align*}
&
\P\big(\tau^{1u}(\alpha)< n_c \big) \le \P(\sigma^{1u}< n_c) = \P\left(\max_{n\le n_c}U_1(n)  \ge a_1 (K+K^{\alpha_{1u}})\right) \le \\
&
\P\left(\max_{n\le n_c}\big((r^+_K)^{-n} U_1(n) -[a_1K]\big )\ge 
[a_1K] ((r^+_K)^{-n_c}-1)+a_1(r^+_K)^{-n_c}K^{\alpha_{1u}}  \right)\stackrel{\dagger}{\le} \\
&
C_1 K^{-\alpha_{1u}}  \E \big|(r^+_K)^{-n_c} U_1(n_c) -[a_1K]\big| \le C_2 K^{-\alpha_{1u}}  \sqrt{\var(U_1(n_c))}  \le  \\
&
C_3 K^{-\alpha_{1u}}  \sqrt{a_1K n_c (r^+_{K})^{2n_c}}  \le C_4 K^{\frac 1 2-\alpha_{1u}}\log_\rho K^c\xrightarrow[K\to\infty]{}0,
\end{align*}
where $\dagger$ holds since $\alpha_{1u} > \alpha_{1\ell}$ and we used Doob's inequality as before. 
This verifies \eqref{PK} and, in turn, \eqref{checkme} for $j=2$. 
The proof for $j=1$ is done similarly and  \eqref{showme22} follows. This completes the proof of \eqref{showme2}.\qed

\begin{acknowledgements}
We are grateful to referees for comments and suggestions that lead to improvement of the paper.
\end{acknowledgements}


\begin{thebibliography}{27}
\providecommand{\natexlab}[1]{#1}
\providecommand{\url}[1]{{#1}}
\providecommand{\urlprefix}{URL }
\expandafter\ifx\csname urlstyle\endcsname\relax
  \providecommand{\doi}[1]{DOI~\discretionary{}{}{}#1}\else
  \providecommand{\doi}{DOI~\discretionary{}{}{}\begingroup
  \urlstyle{rm}\Url}\fi
\providecommand{\eprint}[2][]{\url{#2}}

\bibitem[{Altmann(1993)}]{Altman}
Altmann M (1993) Limits of stochastic epidemics in large populations,
  unpublished manuscript

\bibitem[{Baker et~al.(2018)}]{BCHK}
Baker J, Chigansky P, Hamza K, Klebaner FC (2018) Persistence of small noise
  and random initial conditions. Adv in Appl Probab 50(suppl. A):67--81,
  \doi{10.1017/apr.2018.71},
  \urlprefix\url{https://doi.org/10.1017/apr.2018.71}

\bibitem[{Barbour(1980)}]{B79}
Barbour AD (1980) Density dependent {M}arkov population processes. In:
  Biological growth and spread ({P}roc. {C}onf., {H}eidelberg, 1979), Springer,
  Berlin-New York, Lecture Notes in Biomath., vol~38, pp 36--49

\bibitem[{Barbour et~al.(2015)}]{BHKK15}
Barbour AD, Hamza K, Kaspi H, Klebaner FC (2015) Escape from the boundary in
  {M}arkov population processes. Adv in Appl Probab 47(4):1190--1211,
  \doi{10.1239/aap/1449859806},
  \urlprefix\url{https://doi.org/10.1239/aap/1449859806}

\bibitem[{Barbour et~al.(2016)}]{BCK16}
Barbour AD, Chigansky P, Klebaner FC (2016) On the emergence of random initial
  conditions in fluid limits. J Appl Probab 53(4):1193--1205,
  \doi{10.1017/jpr.2016.74},
  \urlprefix\url{https://doi.org/10.1017/jpr.2016.74}

\bibitem[{Champagnat and M\'{e}l\'{e}ard(2011)}]{CM}
Champagnat N, M\'{e}l\'{e}ard S (2011) Polymorphic evolution sequence and
  evolutionary branching. Probab Theory Related Fields 151(1-2):45--94,
  \doi{10.1007/s00440-010-0292-9},
  \urlprefix\url{https://doi.org/10.1007/s00440-010-0292-9}

\bibitem[{Champagnat et~al.(2002)}]{CFA}
Champagnat N, Ferriare R, Ben~Arous G (2002) The canonical equation of adaptive
  dynamics: A mathematical view. Selection 2(1-2):73--83

\bibitem[{Chigansky et~al.(2018)}]{CJK18}
Chigansky P, Jagers P, Klebaner FC (2018) What can be observed in real time
  {PCR} and when does it show? J Math Biol 76(3):679--695,
  \doi{10.1007/s00285-017-1154-1},
  \urlprefix\url{https://doi.org/10.1007/s00285-017-1154-1}

\bibitem[{Dieckmann and Law(1996)}]{DL96}
Dieckmann U, Law R (1996) The dynamical theory of coevolution: a derivation
  from stochastic ecological processes. J Math Biol 34(5-6):579--612,
  \doi{10.1007/s002850050022},
  \urlprefix\url{https://doi.org/10.1007/s002850050022}

\bibitem[{Freidlin and Wentzell(1984)}]{FW84}
Freidlin MI, Wentzell AD (1984) Random perturbations of dynamical systems,
  Grundlehren der Mathematischen Wissenschaften [Fundamental Principles of
  Mathematical Sciences], vol 260. Springer-Verlag, New York,
  \doi{10.1007/978-1-4684-0176-9},
  \urlprefix\url{https://doi.org/10.1007/978-1-4684-0176-9}, translated from
  the Russian by Joseph Sz\"{u}cs

\bibitem[{Jagers and Klebaner(2011)}]{JKling}
Jagers P, Klebaner FC (2011) Population-size-dependent, age-structured
  branching processes linger around their carrying capacity. J Appl Probab
  48A(New frontiers in applied probability: a Festschrift for S\o ren
  Asmussen):249--260, \doi{10.1239/jap/1318940469},
  \urlprefix\url{https://doi.org/10.1239/jap/1318940469}

\bibitem[{Kendall(1956)}]{Kendall56}
Kendall DG (1956) Deterministic and stochastic epidemics in closed populations.
  In: Proceedings of the {T}hird {B}erkeley {S}ymposium on {M}athematical
  {S}tatistics and {P}robability, 1954--1955, vol. {IV}, University of
  California Press, Berkeley and Los Angeles, pp 149--165

\bibitem[{Kifer(1988)}]{Kif88}
Kifer Y (1988) Random perturbations of dynamical systems, Progress in
  Probability and Statistics, vol~16. Birkh\"{a}user Boston, Inc., Boston, MA,
  \doi{10.1007/978-1-4615-8181-9},
  \urlprefix\url{https://doi.org/10.1007/978-1-4615-8181-9}

\bibitem[{Klebaner(1984)}]{Kleb84}
Klebaner FC (1984) On population-size-dependent branching processes. Adv in
  Appl Probab 16(1):30--55, \doi{10.2307/1427223},
  \urlprefix\url{https://doi.org/10.2307/1427223}

\bibitem[{Klebaner(1993)}]{Kleb93}
Klebaner FC (1993) Population-dependent branching processes with a threshold.
  Stochastic Process Appl 46(1):115--127, \doi{10.1016/0304-4149(93)90087-K},
  \urlprefix\url{https://doi.org/10.1016/0304-4149(93)90087-K}

\bibitem[{Klebaner et~al.(2011)}]{KleVa}
Klebaner FC, Sagitov S, Vatutin VA, Haccou P, Jagers P (2011) Stochasticity in
  the adaptive dynamics of evolution: the bare bones. J Biol Dyn 5(2):147--162,
  \doi{10.1080/17513758.2010.506041},
  \urlprefix\url{https://doi.org/10.1080/17513758.2010.506041}

\bibitem[{Kurtz(1970)}]{Ku70}
Kurtz TG (1970) Solutions of ordinary differential equations as limits of pure
  jump {M}arkov processes. J Appl Probability 7:49--58, \doi{10.2307/3212147},
  \urlprefix\url{https://doi.org/10.2307/3212147}

\bibitem[{Lambert(2005)}]{L05}
Lambert A (2005) The branching process with logistic growth. Ann Appl Probab
  15(2):1506--1535, \doi{10.1214/105051605000000098},
  \urlprefix\url{https://doi.org/10.1214/105051605000000098}

\bibitem[{Lambert(2006)}]{L06}
Lambert A (2006) Probability of fixation under weak selection: A branching
  process unifying approach. Theoretical Population Biology 69(4):419 -- 441,
  \doi{https://doi.org/10.1016/j.tpb.2006.01.002},
  \urlprefix\url{http://www.sciencedirect.com/science/article/pii/S0040580906000116}

\bibitem[{Martin and Lambert(2015)}]{ML}
Martin G, Lambert A (2015) A simple, semi-deterministic approximation to the
  distribution of selective sweeps in large populations. Theoretical Population
  Biology 101:40 -- 46, \doi{https://doi.org/10.1016/j.tpb.2015.01.004},
  \urlprefix\url{http://www.sciencedirect.com/science/article/pii/S0040580915000143}

\bibitem[{Metz et~al.(1996)}]{Metz96}
Metz JA, Geritz SA, Meszena G, Jacobs FJ, Van~Heerwaarden JS (1996) Adaptive
  dynamics: a geometrical study of the consequences of nearly faithful
  reproduction. In: van Strien S, Verduyn~Lunel S (eds) Stochastic and spatial
  structures of dynamical systems, A'dam, North-Holland, pp 183--231

\bibitem[{Metz(1978)}]{Metz}
Metz JAJ (1978) The epidemic in a closed population with all susceptibles
  equally vulnerable; some results for large susceptible populations and small
  initial infections. Acta Biotheoretica 27(1):75--123,
  \doi{10.1007/BF00048405}, \urlprefix\url{https://doi.org/10.1007/BF00048405}

\bibitem[{Mollison(1995)}]{Mollison}
Mollison D (ed)  (1995) Epidemic models: their structure and relation to data,
  vol~5. Cambridge University Press

\bibitem[{Sagitov et~al.(2013)}]{Serik}
Sagitov S, Mehlig B, Jagers P, Vatutin V (2013) Evolutionary branching in a
  stochastic population model with discrete mutational steps. Theoretical
  Population Biology 83:145 -- 154,
  \doi{https://doi.org/10.1016/j.tpb.2012.09.002},
  \urlprefix\url{http://www.sciencedirect.com/science/article/pii/S0040580912000913}

\bibitem[{Seneta(1969)}]{S69}
Seneta E (1969) On {K}oenigs' ratios for iterates of real functions. J Austral
  Math soc 10:207--213

\bibitem[{Shiryaev(1996)}]{Shbook}
Shiryaev AN (1996) Probability, Graduate Texts in Mathematics, vol~95, 2nd edn.
  Springer-Verlag, New York, \doi{10.1007/978-1-4757-2539-1},
  \urlprefix\url{https://doi.org/10.1007/978-1-4757-2539-1}, translated from
  the first (1980) Russian edition by R. P. Boas

\bibitem[{Whittle(1955)}]{Whittle55}
Whittle P (1955) The outcome of a stochastic epidemic---a note on {B}ailey's
  paper. Biometrika 42:116--122

\end{thebibliography}

\end{document}